\documentclass[11pt]{amsart}

\usepackage{amsfonts, amstext, amsmath, amsthm, amscd, amssymb, upgreek}
\usepackage{graphicx, color,  subfigure, wrapfig, overpic, url, wasysym}
\usepackage[dvipsnames]{xcolor}
\usepackage{hyperref}
\usepackage{marginnote}
\long\def\@savemarbox#1#2{\global\setbox#1\vtop{\hsize\marginparwidth 
  \@parboxrestore\tiny\raggedright #2}}
\AtBeginDocument{
   \def\MR#1{}
}

\newcommand{\R}{\mathbb{R}}
\newcommand{\Z}{\mathbb{Z}}

\renewcommand{\H}{\mathbb{H}}

\newcommand{\T}{\mathcal T}

\newcommand{\W}{\mathcal W}
\newcommand{\vol}{{\rm vol}}

\newcommand{\voct}{{v_{\rm oct}}}
\newcommand{\vtet}{{v_{\rm tet}}}
\renewcommand{\L}{\mathcal L}

\newcommand{\G}{\mathcal G}
\newcommand{\m}{\mathrm{m}}
\newcommand{\id}{\mathrm{id}}
\newcommand{\Aut}{\mathrm{Aut}}

\newcommand{\GG}{\langle\Gamma\rangle}
\newcommand{\Q}{\mathcal Q}

\theoremstyle{plain}
\newtheorem{theorem}{Theorem}%[section]
\newtheorem{corollary}{Corollary}%[section]
\newtheorem{lemma}{Lemma}%[section]

\newtheorem{conjecture}{Conjecture}%[theorem]

\newtheorem*{namedtheorem}{\theoremname}
\newcommand{\theoremname}{testing}

\theoremstyle{definition}

\newtheorem{question}{Question}%[theorem]
\newcommand{\refeqn}[1]{equation~\eqref{Eq:#1}}

\title[Examples of homological torsion and volume growth]
  {Examples of homological torsion and volume growth}
\author[A.\ Champanerkar]{Abhijit Champanerkar}
\address{Department of Mathematics, College of Staten Island \& The Graduate Center, City University of New York, New York, NY}
\email{abhijit@math.csi.cuny.edu}
\author[I. \ Kofman]{Ilya Kofman}
\address{Department of Mathematics, College of Staten Island \& The Graduate Center, City University of New York, New York, NY}
\email{ikofman@math.csi.cuny.edu}

\thanks{June 6, 2019.}

\begin{document}

\begin{abstract}
We provide examples of towers of covers of cusped hyperbolic
$3$--manifolds whose exponential homological torsion growth is explicitly
computed in terms of volume growth.  These examples arise from
abelian covers of alternating links in the thickened torus.
A corollary is that the spanning tree entropy for each regular
planar lattice is given by the volume of a hyperbolic polyhedron.
\end{abstract}

\maketitle

%%%%%%%%%%%%%%%%%%%%%%%%%%%%%%%%%%%%%%%%%%%%%%%%%%%%%%%%%%%%%%%%%
\section{Introduction}
For a manifold $M$, a \emph{tower of covers} is a sequence of finite covers,
$$ \cdots \rightarrow M_n \rightarrow M_{n-1} \rightarrow \cdots
\rightarrow M_1 \rightarrow M_0 = M. $$ If $M$ is a $3$--manifold, the
homology groups $H_1(M_n;\Z)$ can have arbitrarily large torsion
subgroups, denoted here by $TH_1(M_n)$.  For a tower of covers of $M$,
the growth rate of the order of $TH_1(M_n)$ has attracted recent
interest in the context of L\"uck's far-reaching Approximation
Conjecture in $L^2$--torsion theory (see \cite{Lueck}).  As a special case
(\cite[Section 7.5]{AFW}, \cite{thang}, \cite[Example~10.5]{Lueck}), if $M$ is a hyperbolic
$3$--manifold with empty or toroidal boundary then, conjecturally, for any tower of regular covers
$M_n$ such that $\bigcap_{n}\pi_1 M_n=\{1\}$,
\begin{equation}\label{Eq:6pi}
  \lim_{n\to\infty}\frac{\log|TH_1(M_n)|}{\vol(M_n)}=\frac{1}{6\pi}.
\end{equation}
With more relaxed conditions on $M$ and $\{M_n\}$, Le \cite{Le2018}
proved $\displaystyle
\limsup_{n\to\infty}\frac{\log|TH_1(M_n)|}{\vol(M_n)}\leq\frac{1}{6\pi}.$ 
\vspace*{0.1cm}

For a knot complement $M=S^3-K$ and finite cyclic covers $X_n$ given by $\pi_1(M)\to\Z/n\Z$, the $L^2$--torsion is well understood.
If $\m(\Delta_K(t))$ denotes the logarithmic Mahler measure of the Alexander polynomial of $K$, 
then the following result proved in \cite{SW2} has been interpreted
(see \cite{Le2014,Raimbault}) as a special case of the Approximation
Conjecture for $\{X_n\}$,
\begin{equation}\label{Eq:SW}
\lim_{n\to\infty}\frac{\log|TH_1(X_n)|}{n}=\m(\Delta_K(t)).
\end{equation}

Our main result, Theorem~\ref{Thm2} below, is a pair of towers of abelian covers of cusped arithmetic
hyperbolic $3$--manifolds, which have totally geodesic surfaces in every cover, such that
$$\lim_{n\to \infty}\frac{\log|TH_1(M_n)|}{\vol(M_n)}=\frac{1}{4\pi}.$$
As far as we know, these are the first examples of towers of non-cyclic covers of
hyperbolic $3$--manifolds whose exponential homological torsion growth
can be computed exactly in terms of volume growth.
These towers of abelian covers are not cofinal as in \refeqn{6pi}. 
They are derived from previous results \cite{ck:det_mp, ckl:mm_voldet, ckp:gmax, ckp:bal_vol}
on alternating links in the thickened torus.

After discussing these two examples, we present a broader context and
a related conjecture for infinite families of links.  Finally, we
prove as a corollary that the spanning tree entropy for each regular
planar lattice is given by the volume of a hyperbolic polyhedron.

\section{Main results}

\begin{figure}[h] 
\begin{tabular}{cccc}
 \includegraphics[height=0.9in]{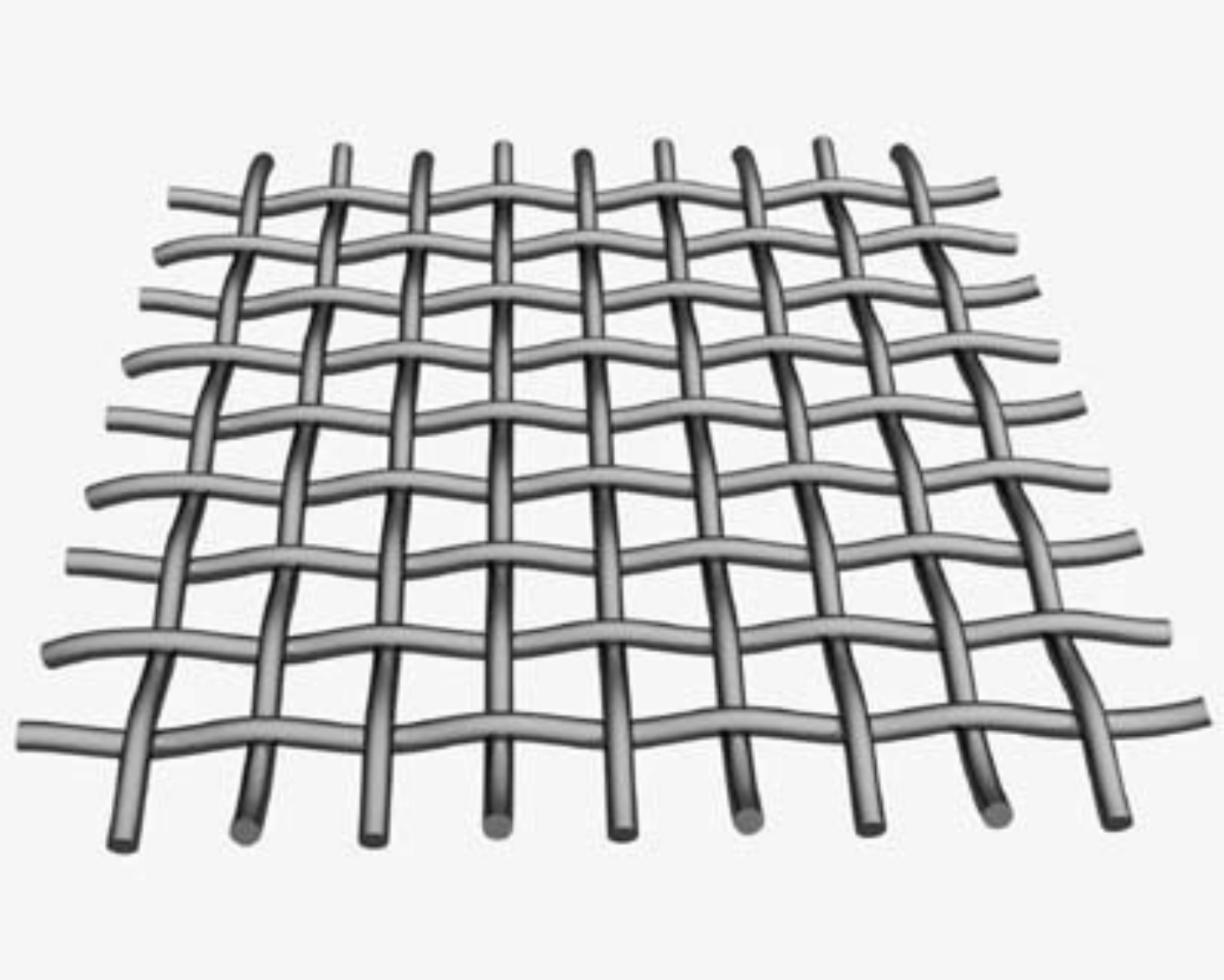} &  \hspace*{0.5cm}
 \includegraphics[height=0.8in]{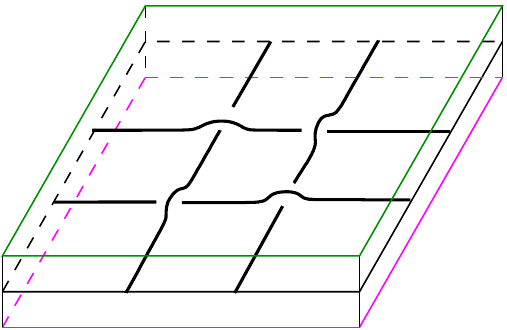} &  \hspace*{0.5cm}
 \includegraphics[height=0.9in]{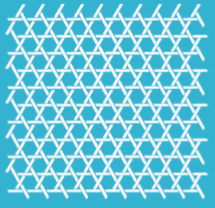} &
\hspace*{0.5cm}
\includegraphics[height=0.9in]{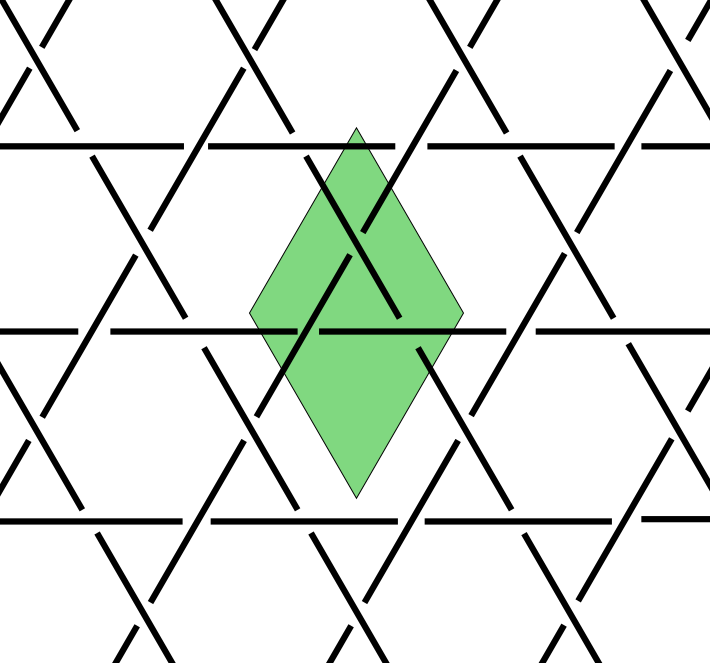} \\
  (a) & \hspace*{0.5cm} (b) & \hspace*{0.5cm} (c) & \hspace*{0.5cm} (d)
\end{tabular}
\caption{(a) Square weave $\W$ in $\R^2\times I$. (b) Link $W_1$ in
  $T^2\times I$. (c) Triaxial link $\Q$ in $\R^2\times I$. (d)
  Fundamental domain for the link $Q_1$ in $T^2\times I$.}
\label{fig:weave}
\end{figure}

\subsection{Biperiodic links} 
A {\em biperiodic alternating link} is a link in $\R^2\times I$, with
$I=(-1,1)$, with a diagram on $\R^2 \times \{0\}$ which can be
isotoped to be invariant under $\Z^2$ acting by independent
translations of $\R^2$, such that the quotient link diagram is
alternating on the torus $T^2 \times \{0\}$.
Figure~\ref{fig:weave}(a) shows the
biperiodic alternating link $\W$, called the square weave, with the
fundamental domain in $T^2\times I$ shown in
Figure~\ref{fig:weave}(b).  Let $W_n$ be the toroidally alternating
link in $T^2\times I$, which is the quotient of $\W$ by the
$n\Z\times n\Z$ action, so that the crossing number $c(W_n)=4n^2$.
Figure~\ref{fig:weave}(c) shows another biperiodic alternating link,
the triaxial link $\Q$.  The fundamental domain for its toroidally
alternating quotient link $Q_1$ in $T^2\times I$ is shown in
Figure~\ref{fig:weave}(d).  Taking the quotient by $n\Z\times n\Z$, we
get the links $Q_n$ in $T^2\times I$ with $c(Q_n)=3n^2$.

For any link $L$ in $T^2\times I$, we have $T^2\times I-L \cong
S^3-\ell$, where $\ell$ has a Hopf sublink given by the cores of the two
Heegaard tori.  In the case of $W_1, \ \ell=L12n2256$
(see~\cite[Figure~4]{ckp:gmax}).  In the case of $Q_1, \ \ell=L12n2232$
(see~\cite[Figure~12]{ckp:bal_vol}), whose complement is homeomorphic to
that of the minimally twisted $5$--chain link.
The complements of both $W_1$ and $Q_1$ are principal congruence manifolds \cite{BGR2018}.

More generally, let $\L$ be any biperiodic alternating link, with a
toroidally alternating quotient link $L_1$ in the fundamental domain
$T^2\times I$.  For $n\geq 1$, taking the quotient by $n\Z\times n\Z$
as above, we get a sequence of toroidally alternating links $L_n$ in
$T^2\times I$.  
Let $X(L_n)$ be the 2--fold cyclic cover of $T^2\times I-L_n$, which
is the cover associated to the kernel of $\displaystyle
\pi_1(T^2\times I-L_n)\to\Z/2\Z$, with each meridian of $L_n$ mapped
to $1$, and the two Hopf link meridians mapped to $0$.
Thus, $\{X(L_n)\}$ are finite covers of $T^2\times I-L_1$, corresponding to finite index subgroups of $\Z^2$.
Moreover, $\R^2\times I-\L$ and, hence, all $T^2\times I-L_n$ and $X(L_n)$ are often
hyperbolic $3$--manifolds; see \cite{ckp:bal_vol}.
In particular, $\W$ and $\Q$ are hyperbolic.

\subsection{Torsion and Alexander polynomials} 
Suppose $T^2\times I-L_1 \cong S^3-\ell$.  The Hopf link meridians
acquire an orientation from the action by $\Z^2$ on the universal
cover of $T^2\times \{0\}$.  Following \cite{sw:graph_complexity}, we
can define the Alexander polynomial
$\Delta_{\ell}(-1,\ldots,-1,x,y)$ via Fox calculus, with
$\Z[\pi_1(S^3-\ell)]\to\Z[x^{\pm 1},y^{\pm 1}]$ given by mapping the
oriented Hopf link meridians to $x,y$, and the other meridians to
$-1$.

Using results in \cite{sw:graph_complexity} and \cite{Dimitrov}, we now extend \refeqn{SW} to these non-cyclic covers.

\begin{lemma}\label{Thm1}
For any biperiodic alternating link $\L$, the manifolds $\{X(L_n)\}$
  form a sequence of covers over the $3$--manifold $X(L_1)$, such that
    $$ \lim_{n\to \infty}\frac{\log|TH_1(X(L_n))|}{n^2}=\m(\Delta_{\ell}(-1,\ldots,-1,x,y)).$$
Passing to the subsequence $n=2^j$ for $j\geq 0$, we get a tower of covers with the same limit.
\end{lemma}

\begin{proof}
Let $G_{\L}$ be the Tait graph (checkerboard graph) of $\L$, and let
$G_{L_1}=G_{\L}/\Z^2$.  The spanning tree entropy of $G_{\L}$ is given
by the logarithmic Mahler measure of the Laplacian determinant
polynomial $\m(D_{G_{L_1}}(x,y))$ (see \cite{lyons,
  sw:graph_complexity}).  As a consequence of the Dehn relations from
the link diagram, the Laplacian matrix is a presentation matrix
(Goeritz matrix) for the homology of the branched 2--fold cyclic
cover, which has the same torsion subgroup $TH_1(X(L_n))$.  Extending
the argument in \cite[Theorem~5.1]{sw:graph_complexity} to two
variables, we obtain
$$\m(D_{G_{L_1}}(x,y))=\m(\Delta_{\ell}(-1,\ldots,-1,x,y)).$$
The limit as claimed is established as $\limsup$ in \cite{sw:graph_complexity},
and the actual limit is proved in \cite{Dimitrov} for cubical sublattices, which in this case are $n\Z\times n\Z\subset\Z^2$.
\end{proof}

For general abelian covers, Le \cite{Le2014,Le2018} proved the following result.
Let $p\colon\tilde X\to X$ be any regular cover of a finite CW-complex with $\Aut(p)=\Z^N$.
Let $\m(\Delta_j(H_j(\tilde X)))$ denote the logarithmic Mahler measure of the first non-zero Alexander polynomial of $H_j(\tilde X;\Z)$.
For $\Gamma\subset\Z^N$ of finite index, let $X_{\Gamma}$ be the corresponding finite cover, and let $\GG=\min\{||{\bf x}||, {\bf x}\in\Gamma-\{0\}\}$.
In \cite{Le2014,Le2018}, Le showed
\begin{equation}\label{Eq:Le}
\limsup_{\GG\to\infty}\frac{\log|TH_1(X_{\Gamma})|}{[\Z^N:\Gamma]}=\m(\Delta_j(H_j(\tilde X))).
\end{equation}
If $N=1$, then Le showed that $\limsup$ can be replaced by $\lim$.
For $N>1$ and cubical sublattices $\Gamma=(n\Z)^N$, Dimitrov \cite{Dimitrov} showed that $\limsup$ can be replaced by $\lim$.
Lemma~\ref{Thm1} provides a version of the latter result for $N=2$.

To prove Theorem~\ref{Thm2}, it will be useful to reinterpret
Lemma~\ref{Thm1} using the toroidal dimer model.  For a biperiodic
alternating link $\L$, the planar balanced bipartite {\em overlaid
  graph} $G_{\L}^b$ was defined in \cite{ck:det_mp,
  ckl:mm_voldet}. The black vertices of $G_{\L}^b$ are vertices of the
Tait graph $G_{\L}$ and its planar dual graph $G_{\L}^*$, the white
vertices of $G_{\L}^b$ are crossings of $\L$, and the edges join each
black vertex to every white vertex incident to that face of $\L$.  Let
$p_{L_1}(z,w)$ be the characteristic polynomial of the toroidal dimer
model on $G_{L_1}^b=G^b_{\L}/\Z^2$, as in \cite{ck:det_mp, ckl:mm_voldet}.

\begin{lemma}\label{entropy}
$$  \lim_{n\to \infty}\frac{\log|TH_1(X(L_n))|}{n^2}=\m(p_{L_1}(z,w)). $$
Passing to the subsequence $n=2^j$ for $j\geq 0$, we get a tower of covers with the same limit,
$$ \cdots \rightarrow X(L_{2n}) \rightarrow X(L_n) \rightarrow \cdots \rightarrow X(L_1). $$
\end{lemma}
\begin{proof}
By \cite[Theorem~3.5]{KOS}, the asymptotic growth rate of the toroidal
dimer model on $G^b_{\L}/(n\Z)^2$ is given by the logarithmic Mahler
measure, $\m(p_{L_1}(z,w))$.
By \cite[Proposition~5.3]{ck:det_mp}, the spanning tree entropy of
$G_{\L}$ equals the growth rate of the toroidal dimer model on $G_{\L}^b$.
Hence,
$$ \m(p_{L_1}(z,w))=\m(D_{G_{L_1}}(x,y))=\m(\Delta_{\ell}(-1,\ldots,-1,x,y)).$$
The claim now follows by Lemma~\ref{Thm1}.
\end{proof}

\subsection{Torsion and Volume}
\begin{theorem}\label{Thm2}
  The manifolds $\{X(W_n)\}$ and $\{X(Q_n)\}$ each form a sequence of covers
  over the cusped arithmetic hyperbolic $3$--manifolds $X(W_1)$ and
  $X(Q_1)$, respectively, such that
  $$ \lim_{n\to \infty}\frac{\log|TH_1(X(W_n))|}{\vol(X(W_n))}=\lim_{n\to \infty}\frac{\log|TH_1(X(Q_n))|}{\vol(X(Q_n))}=\frac{1}{4\pi}. $$
  Passing to the subsequence $n=2^j$ for $j\geq 0$, we get towers of covers with the same limit.
  Moreover, $\{X(W_n)\}$ and $\{X(Q_n)\}$ have embedded totally geodesic surfaces for all $n$.
\end{theorem}

\begin{proof}
Let $\L$ be any hyperbolic biperiodic alternating link, with toroidally alternating quotient links $L_n$ as above.
Since each $X(L_n)$ is a finite cover of $T^2\times I-L_1$,
$$ \vol(X(L_n)) = 2\,\vol(T^2\times I-L_n) = 2n^2\,\vol(T^2\times I-L_1). $$
Hence, the limit in Lemma~\ref{entropy} is equivalent to
\begin{equation}\label{Eq:oldThm1}
\lim_{n\to \infty}\frac{\log|TH_1(X(L_n))|}{\vol(X(L_n))}=\frac{\m(p_{L_1}(z,w))}{2\,\vol(T^2\times I-L_1)}. 
\end{equation}

Specializing to $\W$ and $\Q$, let $\vtet\approx 1.01494$ and
$\voct\approx 3.66386$ be the hyperbolic volumes of the regular ideal
tetrahedron and the regular ideal octahedron, respectively. 
By \cite[Theorems 12, 13]{ckl:mm_voldet}, and adjusting for the fundamental domain in Figure~\ref{fig:weave}(b),
$$ 2\pi\, \m(p_{W_1}) = 4\, \voct \quad \text{and} \quad 2\pi\, \m(p_{Q_1}) = 10\, \vtet. $$
By \cite[Theorem~3.5]{ckp:bal_vol},
$$ \vol(T^2\times I-W_1) = 4\,\voct \quad \text{and} \quad \vol(T^2\times I-Q_1) = 10\,\vtet. $$
The limits now follow as claimed.

We now describe $X(L_n)$ geometrically.  We first decompose $T^2\times
I-L_n$ into two ideal {\em torihedra} $\T^+$ and $\T^-$ as in
\cite[Theorem~2.2]{ckp:bal_vol}. Each torihedron is homeomorphic to
$T^2\times [0,1)$ with finitely many points removed from
  $T^2\times\{0\}$.  If we glue by the identity map along all faces on
  $T^2\times\{0\}$, we have $T^2\times I=\T^+\cup_{\id}\T^-$.  On the
  other hand, if we glue along the checkerboard-colored faces of the
  torihedra on $T^2\times\{0\}$ by homeomorhisms $\phi$ and $\psi$,
  which rotate each shaded face clockwise ($\phi$), and each white
  face counterclockwise ($\psi$), we obtain $T^2\times I-L_n =
  \T^+\cup_{\phi\circ\psi}\T^-$.  Let $\T_1^{\pm}$ and $\T_2^{\pm}$ be
  two copies of each of these torihedra.  Then $X(L_n)$ is obtained as
  follows:
$$X(L_n)=(\T_1^+\cup_{\phi}\T_2^-)\cup_{id}(\T_2^+\cup_{\phi}\T_1^-),$$
where we glue by $\phi$ along shaded faces, and glue by the
identity map along white faces.

By \cite[Theorem~5.1]{ckp:bal_vol}, right-angled torihedra give the complete hyperbolic structure for $T^2\times I-W_n$ and $T^2\times I-Q_n$.
Hence, $X(W_n)$ and $X(Q_n)$ are each obtained by gluing right-angled torihedra.
Thus, by the proof of \cite[Theorem~5.1]{ckp:bal_vol}, $X(W_n)$ and $X(Q_n)$ have totally geodesic checkerboard surfaces.
Finally, arithmeticity follows by \cite[Theorem~4.1]{ckp:bal_vol}.
\end{proof}

For any hyperbolic biperiodic alternating link $\L$, we can construct
the infinite-volume 2--fold cyclic cover $X(\L)$ of $\R^2\times I-\L$
by the same construction as above, replacing each torihedron
$\T^{\pm}$ used in the decomposition of $T^2\times I-L_1$ with its
$\Z^2$--cover, which is homeomorphic to $\R^2 \times [0,1)$, and
gluing faces in the same local pattern as above.  Since $\{X(L_n)\}$
is a sequence of covers, it follows from the definition that its
geometric limit is $X(\L)$ (see \cite{NotesOnNotes}).

The RHS of the conjectured \refeqn{6pi} comes from $L^2$--torsion
theory \cite{Lueck}; namely, for any closed or $1$--cusped hyperbolic
$3$--manifold $M$, the analytic $L^2$--torsion of the corresponding
covering transformations of $\H^3$ is
$$ \rho^{(2)}(M)=-\frac{1}{6\pi}\vol(M).$$
\begin{question}\label{Q1}
Similarly for Theorem~\ref{Thm2}, can $\displaystyle\frac{1}{4\pi}$ be explained in terms of $L^2$--torsion of covering transformations of $X(\W)$ and of $X(\Q)$?
\end{question}

\subsection{Semi-regular links}
The square weave $\W$ and the triaxial link $\Q$ are two examples of a
particularly nice infinite family of hyperbolic \emph{semi-regular links},
studied in \cite{ckp:bal_vol}.  If they do not have bigons, their
planar projections are semi-regular Euclidean tilings, as is the case
for $\W$ and $\Q$.  As a special case of \refeqn{oldThm1}, if $\L$ is
a semi-regular biperiodic alternating link with no bigons, and the
fundamental domain for $L_1$ contains $H$ hexagons and $S$ squares,
then
\begin{equation}\label{Eq:semiregular}
\lim_{n\to \infty}\frac{\log|TH_1(X(L_n))|}{\vol(X(L_n))}=\frac{\m(p_{L_1}(z,w))}{20 H\,\vtet+2 S\,\voct}. 
\end{equation}
Moreover, except for the square weave, $\{X(L_n)\}$ are arithmetic if
and only if the semi-regular Euclidean tiling for $\L$ contains only
triangles and hexagons \cite[Theorem~4.1]{ckp:bal_vol}.

For any hyperbolic biperiodic alternating link $\L$ with toroidally alternating
quotient link $L_1$, conjecturally $\displaystyle \vol(T^2\times
I-L_1)\leq 2\pi\,\m(p_{L_1}(z,w))$ \cite[Conjecture 1]{ckl:mm_voldet}.  We
know of equality occuring only for $\W$ and $\Q$, which are also the
only semi-regular links with totally geodesic checkerboard
surfaces \cite[Theorem~5.1]{ckp:bal_vol}.
Combining the inequality in \cite[Conjecture 1]{ckl:mm_voldet} with
this geometric characterization, we propose the following:

\begin{conjecture}
For any hyperbolic biperiodic alternating link $\L$, with $\{X(L_n)\}$ as above,
  $$ \lim_{n\to \infty}\frac{\log|TH_1(X(L_n))|}{\vol(X(L_n))}\geq\frac{1}{4\pi}, $$
with equality if and only if $\{X(L_n)\}$ have embedded totally geodesic surfaces for all $n$.
\end{conjecture}

\begin{minipage}{4.5in}
  For example, for the Rhombitrihexagonal link $\mathcal R$, with fundamental domain for $R_1$ as shown,
  by \refeqn{semiregular} and \cite[Corollary~14]{ckl:mm_voldet} we have
  $$ \lim_{n\to \infty}\frac{\log|TH_1(X(R_n))|}{\vol(X(R_n))}=\frac{1}{4\pi}\left(\frac{10\,\vtet + 2\pi\,\log(6)}{10\,\vtet + 3\,\voct}\right) \approx \frac{1.0126}{4\pi}. $$
\end{minipage}
\hfill
\begin{minipage}{1.5in}
\begin{center}
 \includegraphics[height=1.2in]{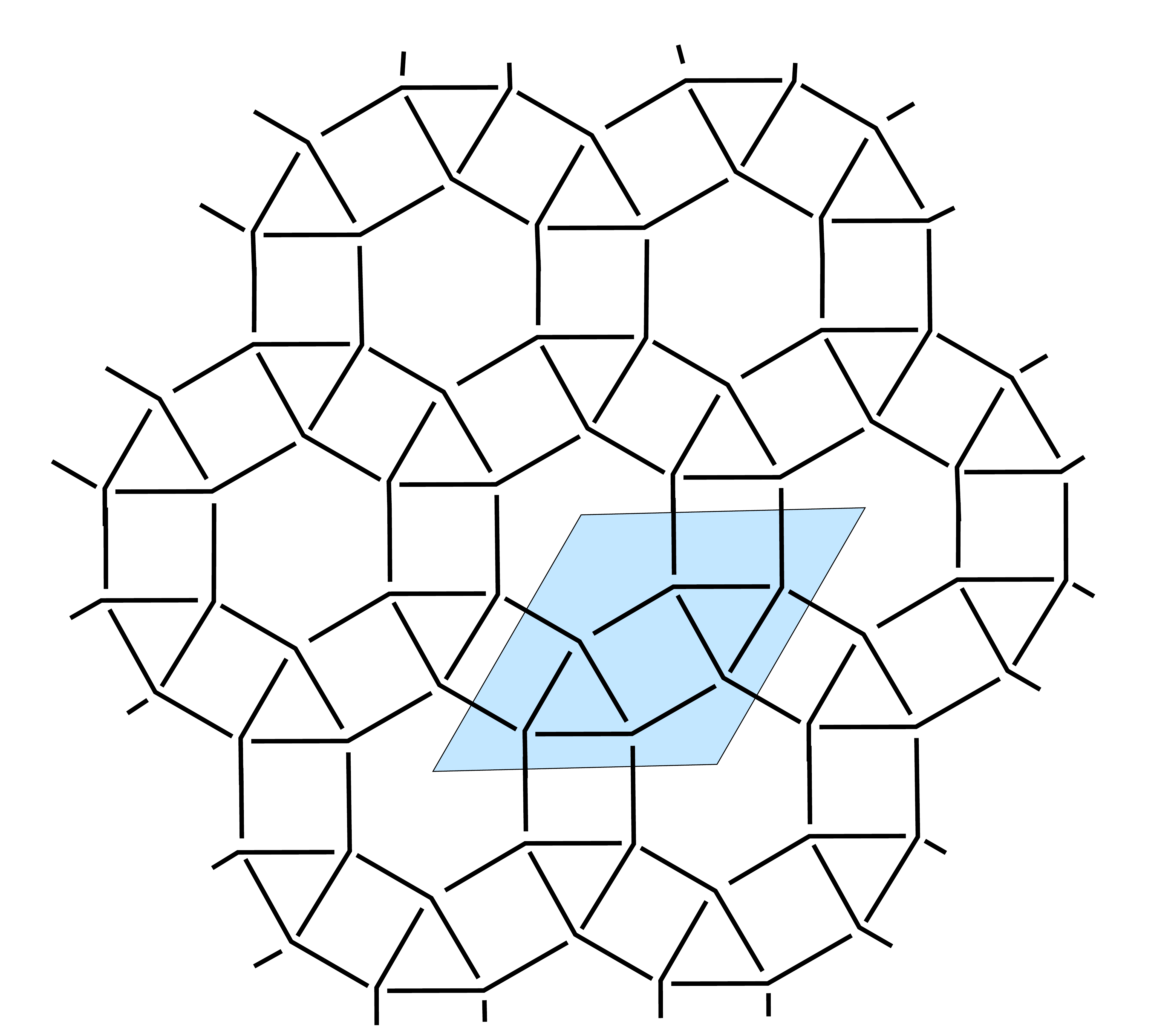} 
\end{center}
\end{minipage}

In~\cite{ckl:mm_voldet}, rigorously computed Mahler measures for
several other biperiodic alternating links also imply that their
corresponding limits are slightly greater than $1/4\pi$.

\section{Spanning tree entropy for regular planar lattices}

Theorem~\ref{Thm2} also provides insight into the surprising fact that each
regular planar lattice has a spanning tree entropy given
by the volume of a hyperbolic polyhedron (cf. \cite{ShrockWu}).
Let $\tau(G)$ be the number of spanning trees of a graph $G$.  Let
$v(G)$ be the number of its vertices.  For a biperiodic lattice $\G$,
let $G_n$ denote the exhaustive nested sequence of finite planar
graphs $\G\cap(n\Z\times n\Z)$.  The spanning tree entropy of $\G$ is
defined as
$$ T(\G) = \lim_{n\to\infty}\frac{\log\tau(G_n)}{v(G_n)}. $$

\begin{corollary}\label{trees}
  The spanning tree entropies for the regular triangular, square and
  hexagonal lattices, $\G^{\triangle},\,\G^{\square}$ and
  $\G^{\hexagon}$, are as follows:
$$ T(\G^{\triangle}) =  10\,\vtet/2\pi, \quad T(\G^{\square}) =   2\,\voct/2\pi, \quad T(\G^{\hexagon}) =  5\,\vtet/2\pi.$$
\end{corollary}
\begin{proof}
  Let $H_n$ denote the toroidal graphs $\G/(n\Z\times n\Z)$, such that
  $v(H_1)=v(G_1)$.  The medial graph of $H_n$ is the projection graph
  of a toroidally alternating link $L_n$ whose Tait graph is $H_n$.
  By Lemma~\ref{entropy} and its proof, the spanning tree entropy of
  $\G$ satisfies:
$$ T(\G) = \lim_{n\to\infty}\frac{\log\tau(G_n)}{v(G_1)\,n^2} = \lim_{n\to \infty}\frac{\log|TH_1(X(L_n))|}{v(H_1)\,n^2}.$$

For the square lattice $\G^{\square}$, we use the square weave $\W$ with $W_1$ as
above, so that $v(H_1^{\square})=2$. For the regular triangular and
hexagonal lattices, $\G^{\triangle}$ and $\G^{\hexagon}$, we use the triaxial link $\Q$ with $Q_1$ as above,
so that $v(H_1^{\triangle})=1$ and $v(H_1^{\hexagon})=2$.
By Theorem~\ref{Thm2} and its proof, their spanning tree entropies are as follows:
$$ T(\G^{\triangle})=\lim_{n\to \infty}\frac{\log|TH_1(X(Q_n))|}{n^2}
=\m(p_{Q_1}(z,w)) =\frac{10\,\vtet}{2\pi}.$$
$$ T(\G^{\square})=\lim_{n\to \infty}\frac{\log|TH_1(X(W_n))|}{2n^2}
= \frac{\m(p_{W_1}(z,w))}{2} =\frac{2\,\voct}{2\pi}.$$
$$ T(\G^{\hexagon})=\lim_{n\to \infty}\frac{\log|TH_1(X(Q_n))|}{2n^2}
= \frac{\m( p_{Q_1}(z,w))}{2} =\frac{5\,\vtet}{2\pi}. $$ 
\end{proof}

The mysterious $2\pi$ factor in Corollary~\ref{trees} seems closely related to Question~\ref{Q1}.

\subsection*{Acknowledgements}
We thank Stefan Friedl, Dan Silver and Susan Williams for useful
discussions.  We thank the organizers of the workshop {\em Growth in
  Topology and Number Theory} at the Hausdorff Center for Mathematics
in Bonn, where this work was started.  We also acknowledge support by
the Simons Foundation.

\bibliographystyle{amsplain} 
\bibliography{references}
\end{document}